\documentclass[a4paper]{article}
\usepackage{graphicx, amsmath, amssymb, amsthm}
\newtheorem{theorem}{Theorem}
\newtheorem{lemma}[theorem]{Lemma}
\theoremstyle{definition}
\newtheorem{ex}{Example}
\def\G{\Gamma}
\def\vep{\varepsilon}

\def\lt{\left}
\def\rt{\right}
\renewcommand{\Re}{\operatorname{\rm Re}}
\renewcommand{\Im}{\operatorname{\rm Im}}

\DeclareMathOperator{\lcm}{lcm}
\newcommand{\set}[1]{\left\{#1\right\}}
\makeatletter
\def\imod#1{\allowbreak\mkern10mu({\operator@font mod}\,\,#1)}
\makeatother

\begin{document}

\title{Coincidence isometries of a shifted square lattice}
 
\author{Manuel Joseph C. Loquias$^{1,2}$ and Peter Zeiner$^1$\\ \small$^1$ Fakult\"at f\"ur Mathematik, Universit\"at Bielefeld, Postfach 100131, 33501 Bielefeld, Germany 
\\ \small$^2$ Institute of Mathematics, University of the Philippines,\\ \small Diliman, C.P. Garcia St., 1101 Diliman, Quezon City, Philippines}
 
 

\maketitle

\begin{abstract}
We consider the coincidence problem for the square lattice that is translated by an arbitrary vector.  General results are obtained about the set of coincidence isometries and the coincidence site lattices of a shifted square lattice by identifying the square lattice with the ring of Gaussian integers.  To illustrate them, we calculate the set of coincidence isometries, as well as generating functions for the number of coincidence site lattices and coincidence isometries, for specific examples.
\end{abstract}

\section{Introduction}
	The sublattice of finite index formed by the points of intersection of a lattice and a rotated copy of the same lattice is called a coincidence site lattice or CSL.  It was Friedel 
	in 1911 who first recognized the use of CSLs in describing and classifying grain boundaries in crystals \cite{Fr}.  Since then, CSLs have proven to be an indispensable tool in the 
	study of grain boundaries and interfaces \cite{KW, R,Bo,P}. This prompted various authors to examine the CSLs of several lattices including cubic and hexagonal ones
	\cite{GBW,G,G2}.

	The discovery of quasicrystals triggered a renewed interest in CSLs.  This led to the analysis of CSLs from a more mathematical point of view.  Known results 
	for lattices were again	considered and reformulated so that they may be readily extended to aperiodic situations.  This was necessary since the first stage in 
	solving the coincidence problem for quasicrystals involved calculating the coincidence site modules (CSMs) of the underlying translation modules, such as 
	modules with 5, 8, 10, and 12-fold symmetry (see \cite{PBR,B} and references therein, see also \cite{W,WL,OWL,WL2}).  Hence, coincidences of lattices and
	modules in dimensions $d\leq 4$ were investigated in \cite{B,Z,Z2,BZ,BGHZ}.  Recent results include the decomposition of coincidence isometries of lattices and
	modules in Euclidean $n$-space as a product of at most $n$ coincidence reflections \cite{Zo,H} and the relationship between the sets of coincidence and
	similarity isometries of lattices and modules \cite{Gl,Gl2}.

	The mathematical treatment of the coincidence problem is very often restricted to linear coincidence isometries, that is, rotations and improper rotations, 
	whereas isometries containing a translational part are ignored -- as we did in the first two paragraphs above.  Nevertheless, general (affine) isometries are 
	important in crystallography.  Indeed, the situation where one shifts the two component crystals against each other has been investigated in \cite{GC,F} and 
	references therein.  It was shown that these shifts are needed to minimize the grain boundary energy, thus they are often referred to as ``rigid 
	relaxations''.  However, some authors claim that minimizing the energy may require shifts that destroy all coincidence sites.

	Even though the idea of introducing a shift after applying a linear coincidence isometry has already been dealt with in the physical literature, not much can be 
	found in the mathematical literature where a systematic treatment of the subject is still missing. Thus, we now aim to generalize the notion of a CSL and CSM, respectively, that is,  
	we investigate the possible intersections of two lattices (modules) that are related by an isometry. To simplify the discussion, we restrict our attention here to the lattice case 
	though most of the results also work in the module case.  In fact, some steps in the general direction have been made in \cite{PBR}.  There, the authors have considered 
	coincidence rotations around certain points which are not lattice (module) points. For instance, they determined the set of coincidence rotations about the center 
	of a Delauney cell of the square lattice and calculated the corresponding indices.

	In this paper we discuss a related and special case: the coincidence problem for shifted lattices.  That is, after translating the lattice by some vector and 
	upon rotation of this shifted lattice (with respect to the origin), we consider its intersection with the shifted lattice. This should be useful in the context of bicrystallography 
	\cite{GP, PV}.  Similar to the approach in \cite{PBR, B}, we start our investigation with the square lattice and provide solutions that, when modified appropriately, also apply to 
	planar modules. 

	The purpose of this paper is to shed further light on the geometry of CSLs.  It is beyond the scope of this paper to discuss the actual grain boundary energy, which would require 
	considering the actual Hamiltonians.  In particular, the paper is not intended to determine which translations are the most favourable ones in terms of energy.

\section{The coincidence problem for lattices}
	Let $\G\subseteq \mathbb{R}^d$ be a $d$-dimensional lattice and $R\in O(d)$.  We say that $R$ is a \emph{(linear) coincidence isometry} of $\G$ if $\G(R):=\G\cap R\G$ is a sublattice 
	of finite index in $\G$ and we call $\G(R)$ a \emph{coincidence site lattice} (CSL) of $\G$.  The \emph{coincidence index} of a coincidence isometry $R$, denoted by $\Sigma(R)$, is 
	given by $\Sigma(R)=[\G:\G(R)]=\frac{\operatorname{vol}(\G(R))}{\operatorname{vol}(\G)}$.  Geometrically, $\Sigma(R)$ is equal to the ratio of the volume of a fundamental domain of 
	$\G$ with the volume of a fundamental domain of $\G(R)$.

	We denote the set of (linear) coincidence isometries of $\G$ by $OC(\G)$ and the set of coincidence rotations of $\G$, that is, $OC(\G)\cap SO(d)$, by $SOC(\G)$. The set $OC(\G)$ 
	forms a group having $SOC(\G)$ as a subgroup \cite{B}.

	We summarize here the known results for the square lattice $\mathbb{Z}^2$ (see \cite{PBR} for details).  The group of coincidence rotations of $\mathbb{Z}^2$ is $SOC(\mathbb{Z}^2)=
	SO(2,\mathbb{Q})$, that is, the coincidence rotations of $\mathbb{Z}^2$ are the special orthogonal matrices having rational entries.  In determining the structure of this group, the
	square lattice is identified with the ring of Gaussian integers $\Gamma=\mathbb{Z}[i]$, where $i=\sqrt{-1}$, embedded in the set of complex numbers $\mathbb{R}[i]=\mathbb{C}$.  In 
	this setting, a coincidence rotation $R$ by an angle of $\theta$ corresponds to multipication by the complex number $e^{i\theta}$, where 
	\begin{equation}\label{coincrot}
		e^{i\theta}=\vep\cdot \prod_{p\equiv 1(4)}{\lt(\frac{\omega_p}{\overline{\omega_p}}\rt)}^{n_p}
	\end{equation}
	with $n_p\in\mathbb{Z}$ and only a finite number of $n_p\neq 0$, $\vep$ is a unit in $\mathbb{Z}[i]$, $p$ runs over the rational primes $p\equiv 1\imod{4}$, 
	and $\omega_p$, and its complex conjugate $\overline{\omega_p}$ are the Gaussian prime factors of $p=\omega_p\cdot\overline{\omega_p}$.  If we denote by $z$ the numerator of 
	$e^{i\theta}$, that is,
	\begin{equation}\label{num}
		z=\prod_{\stackrel{p\equiv 1(4)}{n_p>0}}{\omega_p}^{n_p}\cdot\prod_{\stackrel{p\equiv 1(4)}{n_p<0}}{{\lt(\overline{{\omega_p}}\rt)}^{\,-n_p}},
	\end{equation}
	then the coincidence index of $R$ is the number theoretic norm of $z$, that is, $\Sigma(R)=N(z)=z\cdot\overline{z}$, and the CSL obtained from $R$, $\Gamma(R)$, is the principal ideal
	$(z):=z\mathbb{Z}[i]$.  Consequently, the group of coincidence rotations of the square lattice is given by $SOC(\mathbb{Z}^2)=SOC(\G)\cong C_4\times\mathbb{Z}^{(\aleph_0)}$, where 
	$C_4$ is the cyclic group of order 4 with generator $i$, and $\mathbb{Z}^{(\aleph_0)}$ is the direct sum of countably many infinite cyclic 
	groups each of which is generated by $\omega_p/\overline{\omega_p}$ with $p\equiv 1\imod{4}$.

	Every coincidence isometry $T\in OC(\G)\setminus SOC(\G)$ can be written as $T=R\cdot T_r$, where $R\in SOC(\G)$ and $T_r$ is the reflection along the real axis (complex conjugation).
	Here, $\Sigma(T)=\Sigma(R)$ and $\Gamma(T)=\Gamma(R)$.  Finally, $OC(\mathbb{Z}^2)=OC(\G)\cong SOC(\G)\rtimes C_2$ (semi-direct product), where $C_2$ is the cyclic group of order 2 
	generated by $T_r$.

	The possible coincidence indices and the number of CSLs for a given index $m$ may be described by means of a generating function.  Let $\hat{f}(m)$ be the number of coincidence 
	rotations of $\G$ and $f(m)$ be the number of CSLs of $\G$, for a given index $m$.  Then $\hat{f}(m)=4f(m)$, where the factor 4 stems from the fact that there are four symmetry 
	rotations.  The function $f(m)$ is multiplicative (that is, $f(1)=1$ and $f(mn)=f(m)f(n)$ when $m$, $n$ are relatively prime), and $f(p^r)=2$ for primes $p\equiv 1\imod{4}$ whereas 
	$f(p^r)=0$ for primes $p\equiv 2,3\imod{4}$, where $r\in\mathbb{N}$.  We write the generating function for $f(m)$ as a Dirichlet series $\Phi(s)$ given by
	\begin{alignat*}{2}
		\Phi(s)&=\sum_{m=1}^{\infty}{\frac{f(m)}{m^s}}=\prod_{p\equiv 1(4)}{\frac{1+p^{-s}}{1-p^{-s}}}\\
		&=1+\tfrac{2}{5^s}+\tfrac{2}{13^s}+\tfrac{2}{17^s}+\tfrac{2}{25^s}+\tfrac{2}{29^s}+\tfrac{2}{37^s}+
		\tfrac{2}{41^s}+\tfrac{2}{53^s}+\tfrac{2}{61^s}+\tfrac{4}{65^s}+\tfrac{2}{73^s}+\ldots.
	\end{alignat*}	

\section{Coincidences of shifted lattices}
	We now turn our attention to lattices $\G$ in $\mathbb{R}^d$ that are shifted by some vector $x\in\mathbb{R}^d$ and we look at intersections of the form $(x+\G)\cap R(x+\G)$, where 
	$R\in O(d)$.  We remark that we actually only need to consider values of $x$ in a fundamental domain of $\G$.

	An $R\in O(d)$ is said to be a \emph{(linear) coincidence isometry of the shifted lattice} $x+\G$ if $(x+\G)\cap R(x+\G)$ is a sublattice of $x+\G$ of finite index and we also call 
	$(x+\G)\cap R(x+\G)$ a CSL of the shifted lattice $x+\G$.  We denote the set of all coincidence isometries of $x+\G$ by $OC(x+\G)$.  The following theorem characterizes the set 
	$OC(x+\G)$ and relates the CSLs of $x+\G$ with the CSLs of $\G$ \cite{LZ}.
	
	\begin{theorem}\label{cor3}
		Let $\G$ be a lattice in $\mathbb{R}^d$ and $x\in\mathbb{R}^d$.
		\begin{enumerate}
			
			\item $OC(x+\G)=\set{R\in OC(\G):Rx-x\in\G +R\G}$
			
			\item If $R\in OC(x+\G)$ with $Rx-x=t+Rs$ for some $t,s\in\G$, then \[(x+\G)\cap R(x+\G)=x+(t+\G(R)).\]
		\end{enumerate}
	\end{theorem}
	
	Theorem \ref{cor3} tells us that a CSL of $x+\G$ obtained from $R\in OC(x+\G)$ is just a translate of the CSL $\G(R)$ in $\G$.  Consequently, $[x+\G: (x+\G)\cap R(x+\G)]=\Sigma(R)$
	which means that shifting the lattice does not give rise to new values of coincidence indices. In addition, we see that $OC(x+\G)$ is a subset of $OC(\G)$.  The set 
	$OC(x+\G)$ is non-empty because the identity $1\in OC(x+\G)$.  Also, $OC(x+\G)$ is closed under inverses, that is, $R^{-1}\in OC(x+\G)$ whenever $R\in OC(x+\G)$.  However, given 
	$R_1$, $R_2\in OC(x+\G)$, $R_2R_1$ is not necessarily in $OC(x+\G)$.  In fact, $OC(x+\G)$ is not a group in general \cite{LZ}.

\section{The coincidence problem for the shifted square lattice}
	From this point onwards, we take $\G=\mathbb{Z}[i]$,  the square lattice viewed as the ring of Gaussian integers, and $x\in \mathbb{C}$.  From \eqref{coincrot} and \eqref{num}, we 
	see that we can associate each $R\in SOC(\G)$ to $(z,\vep)$, and we will write this as $R(z,\vep)$.  That is, $R(z,\vep)\in SOC(\G)$ stands for multiplication by the complex number
	$\vep\frac{z}{\overline{z}}$.  We may assume that $\frac{z}{\overline{z}}$ is reduced, that is, $z$ and $\overline{z}$ have no factors in common.  In addition, we shall simply set 
	$z=1$ whenever $R(z,\vep)\in P(\G)$, where $P(\G)$ denotes the point group of $\G$.

	Let $SOC(x+\G):=OC(x+\G)\cap SO(d)$.  We start with the following lemma.
	\begin{lemma}\label{lem3}
		Let $\G=\mathbb{Z}[i]$, $x\in\mathbb{C}$, $R=R(z,\vep)\in SOC(\G)$, and $T=RT_r$.
	 	\begin{enumerate}
	 	 	\item $R\in SOC(x+\G)$ if and only if $(\vep z-\overline{z})x\in\mathbb{Z}[i]$
			
			\item $T\in OC(x+\G)$ if and only if $\vep z\overline{x}-\overline{z}x\in\mathbb{Z}[i]$
		\end{enumerate}
	\end{lemma}
	\begin{proof}
		Recall that $\G$ is a principal ideal domain.  Since $\vep$ is a unit in $\G$ and $z$, $\overline{z}$ are relatively prime,
		\[\G+R\G=\G+\vep\frac{z}{\overline{z}}\G=\frac{1}{\overline{z}}(\overline{z}\G+z\G)=\frac{1}{\overline{z}}\gcd(z,\overline{z})\G=\frac{1}{\overline{z}}\G.\]
		By Theorem \ref{cor3}, $R\in SOC(x+\G)\Leftrightarrow Rx-x\in\G+R\G\Leftrightarrow \vep\frac{z}{\overline{z}}x-x\in\frac{1}{\overline{z}}\G\Leftrightarrow (\vep 
		z-\overline{z})x\in\mathbb{Z}[i].$  Similarly, $\G+T\G=\frac{1}{\overline{z}}\G$. Applying again Theorem \ref{cor3}, we obtain the second statement.
	\end{proof}

	We now obtain the following results about $SOC(x+\G)$ and $OC(x+\G)$.

	\begin{theorem}\label{SOCgroup}
	 	If $\Gamma=\mathbb{Z}[i]$ and $x\in\mathbb{C}$ then $SOC(x+\G)$ is a subgroup of $SOC(\G)$.
	\end{theorem}
	\begin{proof}
	 	We have already mentioned that $1\in SOC(x+\G)$ and $SOC(x+\G)$ is closed under inverses. Let $R_j(z_j,\vep_j)\in SOC(x+\G)$ for $j=1,2$ and $g=\gcd(z_1,\overline{z_2})$.  By 
		Lemma \ref{lem3}, $(\vep_jz_j-\overline{z_j})x\in\mathbb{Z}[i]$ for $j=1,2$.  Write $z_1=h_1g$, $\overline{z_2}=\overline{h_2}g$, and hence, $h_1$ and $\overline{h_2}$ are
		relatively prime.  This means that $R_1R_2$ corresponds to $(h_1h_2,\vep_1\vep_2)$ so that $R_1R_2\in SOC(x+\G)$ if 
		$\lt(\vep_1\vep_2h_1h_2-\overline{h_1}\overline{h_2}\,\rt)x\in\mathbb{Z}[i]$ from Lemma \ref{lem3}.  Now,
		\begin{alignat*}{2}
			(\vep_1\vep_2h_1h_2-\overline{h_1}\overline{h_2}\,)x&=\frac{1}{g}(\vep_1\vep_2z_1h_2-\overline{h_1z_2}\,)x\\
			&=\frac{1}{g}[(\vep_1\vep_2z_1h_2-\vep_2h_2\overline{z_1})+(\vep_2\overline{h_1}z_2-\overline{h_1z_2}\,)]x\\
			&=\frac{1}{g}[\vep_2h_2\underbrace{(\vep_1z_1-\overline{z_1})x}_{\in\;\mathbb{Z}[i]}+\overline{h_1}\underbrace{(\vep_2z_2-\overline{z_2})x}_
			{\in\;\mathbb{Z}[i]}]\in \frac{1}{g}\G.
		\end{alignat*}
		Similarly, we also obtain that $(\vep_1\vep_2h_1h_2-\overline{h_1}\overline{h_2}\,)x\in \frac{1}{\overline{g}}\,\G$.  Hence,
		\[(\vep_1\vep_2h_1h_2-\overline{h_1}\overline{h_2}\,)x\in\frac{1}{g}\G\cap\frac{1}{\overline{g}}\G=\frac{1}{g\overline{g}}(g\G\cap\overline{g}\G)
		=\frac{1}{g\overline{g}}\lcm{(g,\overline{g})}\G=\mathbb{Z}[i]\]
		since $g$, $\overline{g}$ are relatively prime.
	\end{proof}

	For $OC(x+\G)$, the situation is more complicated.  One can show the following results (see \cite{LZ}).

	\begin{theorem}\label{prop5}
		Let $\G=\mathbb{Z}[i]$ and $x\in\mathbb{C}$.
		\begin{enumerate} 
			\item The set $OC(x+\G)$ is a subgroup of $OC(\G)$ if and only if for any $T_1$, $T_2\in$ \mbox{$OC(x+\G)\setminus SOC(x+\G)$}, $T_1T_2\in SOC(x+\G)$.

			\item If $OC(x+\G)$ contains a reflection $T\in P(\G)$ then $OC(x+\G)$ is a subgroup of $OC(\G)$. Also, $OC(x+\G)=SOC(x+\G)\rtimes\langle T\rangle$, where $\langle 
			T\rangle=\set{1,T}\cong C_2$ is the group generated by $T$.
	
			\item Suppose $OC(x+\G)$ does not contain a reflection $T\in P(\G)$.  If $RT_r\in OC(x+\G)$ where $R(z,\vep_1)\in SOC(\G)$ then for any unit $\vep_2$, 
			$R_2=R_2(z,\vep_2)\notin 	SOC(x+\G)$.
		\end{enumerate}
	\end{theorem}
	
	When computing for $OC(x+\G)$, we see from Theorem \ref{prop5} that it is convenient to determine whether there is a reflection $T\in P(\G)$ that is in $OC(x+\G)$.  If such a 
	reflection $T$ exists, then $OC(x+\G)$ is a group and it is the semi-direct product of $SOC(x+\G)$ and $\langle T\rangle$.  Otherwise, we need to check if $RT_r\in OC(x+\G)$
	only for those reflections $RT_r$ for which $R(z,\vep)\in SOC(\G)$ and $R'=R'(z,\vep')\notin SOC(x+\G)$ for all unit $\vep'$ holds.

\section{Specific examples}
	For the rest of the discussion, we shall assume that $R(z,\vep)\in SOC(\G)$.  The following theorem solves completely the case when $x$ has an irrational component 
	\cite{LZ}.

	\begin{theorem}  Let $x=a+bi\in\mathbb{C}$.  If $a$ or $b$ is irrational then $OC(x+\G)$ is a group of at most two elements.  In particular, if
		\begin{enumerate}
			\item $a$ is irrational and $b$ is rational then		
			$OC(x+\Gamma)=\lt\{\begin{aligned}
			\langle T_r\rangle &\;\text{if }2b\in\mathbb{Z}\\
			\set{1} &\;\text{otherwise}.
			\end{aligned}\rt.\;$

			\item $a$ is rational and $b$ is irrational then
			$OC(x+\Gamma)=\lt\{\begin{aligned}
			\langle T\rangle &\;\text{if }2a\in\mathbb{Z}\\
			\set{1} &\;\text{otherwise},
			\end{aligned}\rt.$\\ where $T$ is the reflection along the imaginary axis.
			
			\item both $a$ and $b$ are irrational, and
			\begin{enumerate}
				\item $a$, $b$ are rationally independent then $OC(x+\Gamma)=\set{1}$.
			
				\item $a=\frac{p_1}{q_1}+\frac{p_2}{q_2}b$ where $p_j$, $q_j\in\mathbb{Z}$, $p_j$ and $q_j$ are
				relatively prime (for $j=1,2$) with
				\begin{enumerate}
					\item $p_2q_2$ even, then $OC(x+\Gamma)=\lt\{\begin{aligned}
					\langle R&T_r\rangle &&\text{if }q_1|2q_2\\
					\{&1\} &&\text{otherwise},
					\end{aligned}\rt.$\\where $R=R(p_2+q_2i,1)\in SOC(\G)$.
	
					\item $p_2q_2$ odd, then $OC(x+\Gamma)=\lt\{\begin{aligned}
					\langle R&T_r\rangle &&\text{if }q_1|q_2\\
					\{&1\} &&\text{otherwise},
					\end{aligned}\rt.$\\where $R=R(\frac{p_2+q_2}{2}-\frac{p_2-q_2}{2}i, i)\in SOC(\G)$.
				\end{enumerate}
			\end{enumerate}			
		\end{enumerate}
	\end{theorem}

	\begin{ex}
		\begin{enumerate}
			\item[]		
	
			\item Suppose $x=\frac{1}{\sqrt{2}}+\frac{1}{\sqrt{3}}i$.  We immediately see that we cannot write $\frac{1}{\sqrt{2}}=c+d\frac{1}{\sqrt{3}}$
			where $c,d\in\mathbb{Q}$ since $\sqrt{2}\notin\mathbb{Q}\lt(\sqrt{3}\,\rt)$.  Hence, $OC(x+\G)=\set{1}$.
		
			\item Let $x=\sqrt{2}-\frac{\sqrt{2}}{2}i$.  We have $\sqrt{2}=\frac{0}{1}+\lt(\frac{-2}{1}\rt)\lt(-\frac{\sqrt{2}}{2}\rt)$, and since $1|(2\cdot 1)$, 
			$OC(x+\Gamma)=\langle RT_r\rangle$ where $R=R(-2+i,1)\in SOC(\G)$.			
		\end{enumerate} 
	\end{ex}

	It only remains to consider the case when both $a$ and $b$ are rational.  We now consider $x=a+bi\in\mathbb{Q}(i)$ and write $x=\frac{p}{q}$ where $p$, $q\in\mathbb{Z}[i]$,
	and $p$, $q$ are relatively prime (in $\mathbb{Z}[i]$). The following lemma tells us that $SOC(x+\G)$ depends only on the denominator $q$ of $x$.

	\begin{lemma}\label{prop8} 
		Let $\G=\mathbb{Z}[i]$, $x=\frac{p}{q}\in\mathbb{Q}(i)$ where $p$, $q\in\mathbb{Z}[i]$, with $p$, $q$ relatively prime, and $R(z,\vep)\in SOC(\G)$.  Then
		$R\in SOC(x+\G)$ if and only if $q$ divides $\vep z-\overline{z}$. Furthermore, $SOC(x+\G)=SOC(\frac{1}{q}+\G)$.
	\end{lemma}
	\begin{proof}
	 	We know from Lemma \ref{lem3} that if $R\in SOC(x+\G)$ then $(\vep z-\overline{z})x=\frac{(\vep z-\overline{z})p}{q}\in\mathbb{Z}[i]$.  Since $p$ and $q$ are relatively 
		prime, $q|(\vep z-\overline{z})$.  The second statement follows from the first.
	\end{proof}

	\begin{lemma}\label{lemfunddom}
		Suppose $\Gamma=\mathbb{Z}[i]$ and $x\in\mathbb{C}$.  If $x'=Qx$ for some $Q\in P(\G)$ then \[OC(x'+\G)=Q[OC(x+\G)]Q^{-1}.\]	
	\end{lemma}
	\begin{proof}
		Since $SOC(\G)$ is a normal subgroup of $OC(\G)$, $R\in SOC(\G)$ if and only if $QRQ^{-1}\in SOC(\G)$.  Thus, if $R\in SOC(\G)$ then it follows from Theorem \ref{cor3} that
		\begin{alignat*}{2}
			R\in OC(x'+\G)&\Leftrightarrow Rx'-x'\in\G+R\G\\
			&\Leftrightarrow Q(Q^{-1}RQx-x)\in Q(\G+Q^{-1}RQ\G)\\
			&\Leftrightarrow Q^{-1}RQx-x\in \G+Q^{-1}RQ\G\\
			&\Leftrightarrow Q^{-1}RQ\in OC(x+\G)\\
			&\Leftrightarrow R\in Q[OC(x+\G)]Q^{-1}.\qedhere
		\end{alignat*}
	\end{proof}	

	Recall that we only need to consider values of $x=a+bi$ in a fundamental domain of $\Gamma$.  A fundamental domain of $\G$ is $\set{a+bi\in\mathbb{C}:-\frac{1}{2}\leq 
	a,b<\frac{1}{2}}$ (see Figure \ref{funddom}).  Observe that every point $x'$ in the chosen fundamental domain can be written as $x'=Qx$ where $Q\in P(\G)$ and
	$x\in\set{a+bi\in\mathbb{C}: 0\leq b\leq a\leq\frac{1}{2}}$ (a fundamental domain of the symmetry group of $\G$ which is a crystallographic group of type $p4m$).  Hence, it follows 
	from Lemma \ref{lemfunddom} that we only need to compute $OC(x+\G)$ for values of $x=a+bi$, where $0\leq b\leq a\leq\frac{1}{2}$ (see Figure \ref{funddom}).

	\begin{figure}[ht]
		\setlength{\unitlength}{1.5in}
		\begin{center}
			\begin{picture}(1.3,1.3)(-0.65,-0.65)
				\linethickness{0.15pt}
				\put(-0.65,0){\vector(1,0){1.3}}
				\put(0,-0.65){\vector(0,1){1.3}}
				\put(0.005,-0.60){-$\frac{1}{2}$}
				\put(-0.6,-0.1){-$\frac{1}{2}$}				
				\put(0.02,0.55){$\frac{1}{2}$}
				\put(0.52,-0.1){$\frac{1}{2}$}
				\put(-0.5,-0.5){\line(1,0){1}}
				\put(-0.5,-0.5){\line(0,1){1}}
				\multiput(-0.5,0.5)(0.105,0){10}{\line(1,0){0.055}}
				\multiput(0.5,0.5)(0,-0.105){10}{\line(0,-1){0.055}}
				\thicklines
				\put(0,0){\line(1,0){0.5}}
				\put(0.5,0){\line(0,1){0.5}}
				\put(0,0){\line(1,1){0.5}}
			\end{picture}
		\end{center}
		\caption{A fundamental domain of $\G$, or unit cell, and a fundamental domain of the symmetry group of $\G$ (black triangle)}\label{funddom}
	\end{figure}
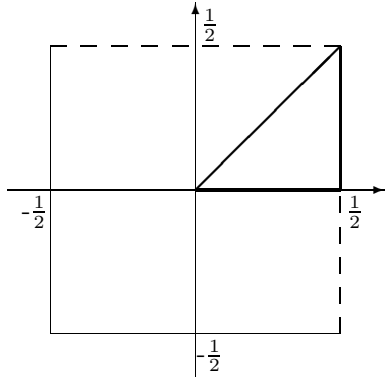	

	\begin{lemma}\label{lemgroup}
	 	If $\G=\mathbb{Z}[i]$ and $x=a+bi\in\mathbb{C}$, then $OC(x+\G)$ is a subgroup of $OC(\G)$ if one of the following conditions is satisfied: $a\in\frac{1}{2}\mathbb{Z}$, 
		$b\in\frac{1}{2}\mathbb{Z}$ or $a\pm b\in\mathbb{Z}$.  Furthermore, $OC(x+\G)=SOC(x+\G)\rtimes \langle RT_r\rangle$ where $R=R(1,\vep)\in SOC(\G)$, and 
		\[\vep=\lt\{\begin{aligned}
					1 & \;\text{if }b\in\tfrac{1}{2}\mathbb{Z}\\
					-1 & \;\text{if }a\in\tfrac{1}{2}\mathbb{Z}\\
					i & \;\text{if }a-b\in\mathbb{Z}\\
					-i & \;\text{if }a+b\in\mathbb{Z}.
				\end{aligned}\rt.\]		
	\end{lemma}
	\begin{proof}
	 	Consider the reflection $RT_r\in P(\G)$.  By Lemma \ref{lem3}, $RT_r\in OC(x+\G)$ if and only if $\vep\overline{x}-x\in\mathbb{Z}[i]$.  The result follows by applying Theorem 
		\ref{prop5}.
	\end{proof}

	In particular, for values of $a$ and $b$ for which $0\leq b\leq a\leq \frac{1}{2}$, $OC(x+\G)$ is a subgroup of $OC(\G)$ when $a=\frac{1}{2}$, $b=0$, or $a=b$ (the boundaries of the 
	triangle in Figure \ref{funddom}).
	
	Before looking at some examples, we note that given $R(z,\vep)\in SOC(\G)$, we have
	\begin{equation}\label{ezmcz}
	\vep z-\overline{z}=\lt\{\begin{aligned}
		2\,&\Im(z) &&\;\text{if }\vep=1\\
		-2\,&\Re(z) &&\;\text{if }\vep=-1\\
		-[\Re(z)+&\Im(z)](1-i) &&\;\text{if }\vep=i\\
		-i[\Re(z)-&\Im(z)](1-i) &&\;\text{if }\vep=-i\;.
	\end{aligned}\rt.
	\end{equation}
	In addition, we see from \eqref{coincrot} and \eqref{num} that $\Re(z)$ and $\Im(z)$ are relatively prime and of different parity (that is, one is odd and the other is even).

	We will also exhibit the number of possible coincidence rotations and CSLs obtained with given index $m$ of the shifted lattice $x+\G$ by means of generating functions.  We shall 
	denote by $\hat{f}_x(m)$ the number of coincidence rotations of $x+\Gamma$ of index $m$, and $f_x(m)$ the number of CSLs of $x+\Gamma$ of index $m$.
	
	\begin{ex} $x=\frac{1}{2}+\frac{1}{2}i=\frac{1}{1-i}$
		\setlength{\unitlength}{1.25in}
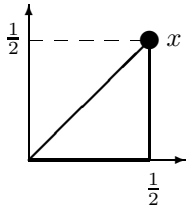
		\begin{center}
			\begin{picture}(0.65,0.80)(0,-0.15)
				\linethickness{0.10pt}
				\put(0,0){\vector(1,0){0.65}}
				\put(0,0){\vector(0,1){0.65}}
				\put(-0.1,0.48){$\frac{1}{2}$}
				\put(0.48,-0.15){$\frac{1}{2}$}
				\multiput(0,0.5)(0.105,0){5}{\line(1,0){0.055}}
				\thicklines
				\put(0,0){\line(1,0){0.5}}
				\put(0.5,0){\line(0,1){0.5}}
				\put(0,0){\line(1,1){0.5}}
 				\put(0.5,0.5){\circle*{0.08}}
				\put(0.6,0.5){\makebox(0,0){$x$}}
			\end{picture}
		\end{center}
		
		The denominator of $x$ is $q=1-i$ and we see from \eqref{ezmcz} that $q|(\vep z-\overline{z})$ for all $z$, $\vep$.  Lemmas \ref{prop8} and \ref{lemgroup} implies that 
		$SOC(x+\G)=SOC(\G)\cong	C_4\times{\mathbb{Z}}^{(\aleph_0)}$ and $OC(x+\G)=OC(\G)\cong SOC(x+\G)\rtimes C_2$.  Clearly, $OC(x+\G)$ is a subgroup of $OC(\G)$ in this case.  
		Also, $\hat{f}_x(m)=\hat{f}(m)$ and $f_x(m)=f(m)$.  

		These results agree with the results obtained in the Appendix of \cite{PBR}
		(just shift the center of the Delaunay cell into the origin).
		
		We also note here that $(S)OC(x+\G)=(S)OC(\G)$ if and only if $x=\frac{m}{2}+\frac{n}{2}i$, where $m$, $n$ are odd integers.  
	\end{ex}
	
	\begin{ex} $x=\frac{1}{2}$
		\setlength{\unitlength}{1.25in}
		\begin{center}
			\begin{picture}(0.65,0.80)(0,-0.15)
				\linethickness{0.10pt}
				\put(0,0){\vector(1,0){0.65}}
				\put(0,0){\vector(0,1){0.65}}
				\put(-0.1,0.48){$\frac{1}{2}$}
				\put(0.46,-0.18){$\frac{1}{2}$}
				\multiput(0,0.5)(0.105,0){5}{\line(1,0){0.055}}
				\thicklines
				\put(0,0){\line(1,0){0.5}}
				\put(0.5,0){\line(0,1){0.5}}
				\put(0,0){\line(1,1){0.5}}
 				\put(0.5,0){\circle*{0.08}}
				\put(0.57,0.07){\makebox(0,0){$x$}}
			\end{picture}
		\end{center}
		
		The denominator of $x$ is $q=2$.  Since the sum of $\Re(z)$ and $\Im(z)$ is odd, we obtain from \eqref{ezmcz} that for all $z$, $q|(\vep z-\overline{z})$ if and only if 
		$\vep=\pm 1$.  Hence, by Lemma \ref{prop8}, \[SOC(x+\G)=\set{R(z,\vep)\in SOC(\G):\vep=\pm 1}\cong C_2\times{\mathbb{Z}}^{(\aleph_0)}.\]  From Lemma \ref{lemgroup}, 
		$OC(x+\G)=SOC(x+\G)\rtimes\langle T_r\rangle$ and is a subgroup of $OC(\G)$ of index 2.  In this case, we have $f_x(m)=f(m)$ but $\hat{f}_x(m)=2f_x(m)$.
	\end{ex}

	\begin{ex} $x_0=\frac{1}{3}$ and $x_1=\frac{1}{3}+\frac{1}{3}i$
		\setlength{\unitlength}{1.25in}
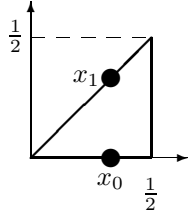
		\begin{center}
			\begin{picture}(0.65,0.80)(0,-0.15)
				\linethickness{0.10pt}
				\put(0,0){\vector(1,0){0.65}}
				\put(0,0){\vector(0,1){0.65}}
				\put(-0.1,0.48){$\frac{1}{2}$}
				\put(0.46,-0.18){$\frac{1}{2}$}
				\multiput(0,0.5)(0.105,0){5}{\line(1,0){0.055}}
				\thicklines
				\put(0,0){\line(1,0){0.5}}
				\put(0.5,0){\line(0,1){0.5}}
				\put(0,0){\line(1,1){0.5}}
 				\put(0.33,0){\circle*{0.08}}
				\put(0.33,-0.1){\makebox(0,0){$x_0$}}
				\put(0.33,0.33){\circle*{0.08}}
				\put(0.23,0.33){\makebox(0,0){$x_1$}}					
			\end{picture}
		\end{center}

		Both $x_0$ and $x_1$ have denominator $q=3$.  It is easy to see that if both rational integers $m$, $n$ are not divisible by $q$, then either $m+n$ or $m-n$ is divisible by 
		$q$.  Hence, by \eqref{ezmcz}, there is a unique $\vep$ so that $q|(\vep z-\overline{z})$ for all $z$. We conclude from Lemma \ref{prop8} that 
		$SOC(x_j+\G)\cong{\mathbb{Z}}^{(\aleph_0)}$ for $j=0,1$.  In addition, by Lemma \ref{lemgroup}, $OC(x_j+\G)=SOC(x_j+\G)\rtimes\langle RT_r\rangle$, where $R=R(1,i^j)\in 
		P(\G)$, and $OC(x_j+\G)$ is a subgroup of $OC(\G)$ of index 4.  Finally, we have $\hat{f}_{x_j}(m)=f_{x_j}(m)=f(m)$.
	\end{ex}

	\begin{ex} $x_0=\frac{1}{5}$, $\frac{2}{5}$ and $x_1=\frac{1}{5}+\frac{1}{5}i$, $\frac{2}{5}+\frac{2}{5}i$
		\setlength{\unitlength}{1.25in}
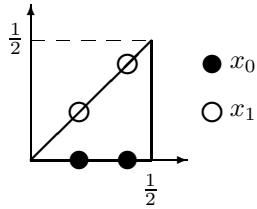
		\begin{center}
			\begin{picture}(0.65,0.8)(0,-0.15)
				\linethickness{0.10pt}
				\put(0,0){\vector(1,0){0.65}}
				\put(0,0){\vector(0,1){0.65}}
				\put(-0.1,0.48){$\frac{1}{2}$}
				\put(0.46,-0.15){$\frac{1}{2}$}
				\multiput(0,0.5)(0.105,0){5}{\line(1,0){0.055}}
				\thicklines
				\put(0,0){\line(1,0){0.5}}
				\put(0.5,0){\line(0,1){0.5}}
				\put(0,0){\line(1,1){0.5}}
				\put(0.2,0){\circle*{0.08}}
				\put(0.4,0){\circle*{0.08}}
				\put(0.2,0.2){\circle{0.08}}
				\put(0.4,0.4){\circle{0.08}}
				\put(0.75,0.4){\circle*{0.08}}
				\put(0.75,0.2){\circle{0.08}}
				\put(0.88,0.4){\makebox(0,0){$x_0$}}					
				\put(0.88,0.2){\makebox(0,0){$x_1$}}					
			\end{picture}
		\end{center}

		We have the denominator $q=5$ for $x_j$, $j=0$, $1$.  By considering each possible combination of $\Re(z)$ and $\Im(z)$ modulo $q$, it can be verified that for all $z$ with 
		$5\nmid N(z)$, there is a unique $\vep$ such that $R(z,\vep)\in SOC(\frac{1}{5}+\G)$.  This means that $SOC(x_j+\G)\cong{\mathbb{Z}}^{(\aleph_0)}$ for $j=0,1$ by Lemma 
		\ref{prop8}.  Also, it follows from Lemma \ref{lemgroup} that $OC(x_j+\G)$ is a subgroup of $OC(\G)$ and $OC(x_j+\G)=SOC(x_j+\G)\rtimes\langle RT_r\rangle$, where 
		$R=R(1,i^j)\in P(\G)$.  Furthermore, $\hat{f}_{x_j}(m)=f_{x_j}(m)$ where the Dirichlet series generating function for $f_{x_j}(m)$ is given by
		\begin{alignat*}{2}
			\Phi_{x_j}(s)&=\sum_{m=1}^{\infty}{\frac{f_{x_j}(m)}{m^s}}=\prod_{\stackrel{p\equiv 1(4)}{p\neq 5}}\frac{1+p^{-s}}{1-p^{-s}}\\
							&=1+\tfrac{2}{13^s}+\tfrac{2}{17^s}+\tfrac{2}{29^s}+\tfrac{2}{37^s}+\tfrac{2}{41^s}+\tfrac{2}{53^s}+\tfrac{2}{61^s}+\tfrac{2}{73^s}
							+\tfrac{2}{89^s}+\tfrac{2}{97^s}+\tfrac{2}{101^s}+\tfrac{2}{109^s}+\tfrac{2}{113^s}+\\
							&\quad\tfrac{2}{137^s}+ \tfrac{2}{149^s}+\tfrac{2}{157^s}+\tfrac{2}{169^s}+\tfrac{2}{173^s}+\tfrac{2}{181^s}+
							\tfrac{2}{193^s}+\tfrac{2}{197^s}+\tfrac{4}{221^s}+\tfrac{2}{229^s}+\ldots.\,.
		\end{alignat*}
	\end{ex}

	\begin{ex} $x=\frac{2}{5}+\frac{1}{5}i=\frac{i}{1+2i}$
		\setlength{\unitlength}{1.25in}
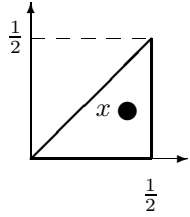
		\begin{center}
			\begin{picture}(0.65,0.80)(0,-0.15)
				\linethickness{0.10pt}
				\put(0,0){\vector(1,0){0.65}}
				\put(0,0){\vector(0,1){0.65}}
				\put(-0.1,0.48){$\frac{1}{2}$}
				\put(0.46,-0.18){$\frac{1}{2}$}
				\multiput(0,0.5)(0.105,0){5}{\line(1,0){0.055}}
				\thicklines
				\put(0,0){\line(1,0){0.5}}
				\put(0.5,0){\line(0,1){0.5}}
				\put(0,0){\line(1,1){0.5}}
 				\put(0.4,0.2){\circle*{0.08}}
				\put(0.3,0.2){\makebox(0,0){$x$}}				
			\end{picture}
		\end{center}	

		Since the denominator $q=1+2i$ of $x$ does not divide $1-i$, we see from \eqref{ezmcz} that $q|(\vep z-\overline{z})$ if and only if $5|(\vep z-\overline{z})$.  Hence, 
		$SOC(x+\G)=SOC(\frac{1}{5}+\G)\cong{\mathbb{Z}}^{(\aleph_0)}$ by Lemma \ref{prop8}. Applying Theorem \ref{prop5}, if $RT_r\in OC(x+\G)$ where $R(z,\vep)\in SOC(\G)$ then 
		$5|N(z)$ because $OC(x+\G)$ does not contain a reflection in $P(\G)$.  Observe that given $z$ with $5|N(z)$, we must find either $1+2i$ or $1-2i$ (and not both) in the 
		factorization of $z$ into primes in $\mathbb{Z}[i]$.  If $(1-2i)|z$ then $z\overline{x}\in\mathbb{Z}[i]$ and $\vep z\overline{x}-\overline{z}x=\vep 
		z\overline{x}-\overline{z\overline{x}}\in\mathbb{Z}[i]$, $\forall\vep$.  Lemma \ref{lem3} then implies that
		\[OC(x+\G)=SOC(x+\G)\cup\set{RT_r: R(z,\vep)\in SOC(\G)\text{ with }(1-2i)|z}.\]
		Here, we have an example of a set $OC(x+\G)$ that is not a group.  Indeed, let $T_k=R_kT_r\in OC(x+\G)\setminus SOC(x+\G)$ where $R_k=R_k(z,\vep_k)\in SOC(\G)$, for $k=1,2$, 
		with $\vep_1\neq\vep_2$.  We obtain that $T_1T_2$ is not the identity with $T_1T_2=R_1{R_2}^{-1}\in P(\G)$ which means that $T_1T_2\notin OC(x+\G)$.  By Theorem \ref{prop5}, 
		$OC(x+\G)$ is not a subgroup of $OC(\G)$.

		The Dirichlet series generating function for $f_{x}(m)$ is given by
		\begin{alignat*}{2}
			\Phi_{x}(s)&=\sum_{m=1}^{\infty}{\frac{f_{x}(m)}{m^s}}=\frac{1}{1-5^{-s}}\cdot\prod_{\stackrel{p\equiv 1(4)}{p\neq 5}}\frac{1+p^{-s}}{1-p^{-s}}\\
			&=1+\tfrac{1}{5^s}+\tfrac{2}{13^s}+\tfrac{2}{17^s}+\tfrac{1}{25^s}+\tfrac{2}{29^s}+\tfrac{2}{37^s}+\tfrac{2}{41^s}+\tfrac{2}{53^s}+\tfrac{2}{61^s}+\tfrac{2}{65^s}+
			\tfrac{2}{73^s}+\ldots\,.					
		\end{alignat*}
		Also, if we denote by $\hat{F}_x(m)$ the number of (linear) coincidence isometries of $x+\Gamma$ of index $m$, we obtain that
		\[\hat{F}_{x}(m)=\lt\{\begin{aligned}
					f_{x}(m) &\;\text{if } 5 \nmid\, m\\
					4\,f_{x}(m) &\;\text{if } 5\,|\, m.\\
					\end{aligned}\rt.\]
		Hence, the Dirichlet series generating function for $\hat{F}_{x}(m)$ is given by
		\begin{alignat*}{2}
			\Psi_{x}(s)&=\sum_{m=1}^{\infty}{\frac{\hat{F}_{x}(m)}{m^s}}=\frac{1+3\cdot 5^{-s}}{1-5^{-s}}\cdot\prod_{\stackrel{p\equiv 1(4)}{p\neq 5}}\frac{1+p^{-s}}{1-p^{-s}}\\
			&=1+\tfrac{4}{5^s}+\tfrac{2}{13^s}+\tfrac{2}{17^s}+\tfrac{4}{25^s}+\tfrac{2}{29^s}+\tfrac{2}{37^s}+\tfrac{2}{41^s}+\tfrac{2}{53^s}+\tfrac{2}{61^s}+\tfrac{8}{65^s}+
			\tfrac{2}{73^s}+\ldots\,.					
		\end{alignat*}		
	\end{ex}	

\section{Conclusion and outlook}
We have seen that the coincidence isometries of a shifted lattice are also coincidence isometries of the original lattice.  Moreover, the CSLs of the shifted lattice are merely translations of CSLs of the original lattice.  Thus, no new values of coincidence indices $\Sigma$ are obtained by shifting the lattice, and some $\Sigma$-values even disappear or their multiplicity is reduced.

The coincidences of a shifted square lattice were examined in this paper by identifying the lattice with the ring of Gaussian integers.  The problem was completely solved for the case when
the shift consists of an irrational component.  For the remaining case, that is, when the shift may be written as a quotient of two Gaussian integers that are relatively prime, one needs
to compute the set of coincidence rotations for each possible denominator  via some divisibility condition.  Partial results are given here and in \cite{LZ} on how to obtain the coincidence isometries and indices for any given denominator and corresponding numerator.  General results in this direction will depend on the arithmetic of the Gaussian integers.

It should be emphasized that the order of rotation and translation of a lattice is in general not interchangeable.  In this paper, we compare the shifted lattice with its rotated copy, that is, the translation (say $x$) comes first before rotation.  This corresponds to the situation where the lattice is first rotated by $R$ and is shifted afterwards by the vector $Rx-x$, which is equivalent to a rotation of the lattice about a different point ($-x$), thus keeping at least one point ($-x$) fixed.    In particular, the CSLs of a shifted lattice are shifted copies of the intersection of a lattice with a rotated, then translated version of the same lattice (see \cite{LZ}).

The next step is to extend these results to planar modules by identifying the modules with rings of cyclotomic integers (\cite{PBR}).  General results for lattices will of course also hold in three dimensions, but the approach in this case will not be via complex numbers but via quaternions.  

Finally, it is expected that the ideas behind the study of a shifted lattice may be applied to crystals where there is more than one atom per primitive unit cell, as described in \cite{GP,PV}.  A related mathematical problem is the following: Suppose the coincidence problem for a sublattice (of finite index) of a given lattice has already been solved. What can be deduced about the coincidence indices of the the original lattice?  A possible approach to answer this question involves looking at the coincidences of the corresponding cosets, which are just shifted copies of the sublattice.

\subsection*{Acknowledgements}
The authors are grateful to the referee for his valuable remarks on the manuscript. M. Loquias would like to thank the Deutscher Akademischer Austausch Dienst (DAAD) for financial support during his stay in Germany. This work was supported by the German Research Council (DFG), within the CRC~701.

\end{document}